\documentclass{amsart}

\usepackage{amsmath}
\usepackage{amsfonts}
\usepackage{url}

\usepackage[hcentering,margin=1.5in]{geometry}

\usepackage{amsthm}
\newtheorem{theorem}{Theorem}

\newtheorem{prop}[theorem]{Proposition}
\theoremstyle{definition}
\newtheorem*{constra}{Construction A}

\newcommand\C{\mathbb{C}}
\newcommand\F{\mathbb{F}}
\newcommand\R{\mathbb{R}}
\newcommand\Z{\mathbb{Z}}
\newcommand\wt{\mathrm{wt}}

\newcommand\inprod[2]{\langle#1,#2\rangle}
\newcommand\latinprod[2]{(#1,#2)}

\usepackage{verbatim}

\title[Type~II Codes of Length~$24$]{On the
Classification of Type~II Codes of Length~$24$}

\author[N.\ D.\ Elkies]{Noam D.\ Elkies}
\address{Department of Mathematics, Harvard University\newline\indent
One Oxford Street\newline \indent Cambridge, MA 02138}
\email{elkies@math.harvard.edu}

\author[S.\ D.\ Kominers]{Scott D.\ Kominers}
\address{Department of Mathematics, Harvard University\newline\indent c/o
8520 Burning Tree Road\newline \indent Bethesda, MD 20817}
\email{kominers@fas.harvard.edu, skominers@gmail.com}


\begin{document}
\begin{abstract}
We give a new, purely coding-theoretic proof of Koch's criterion
on the tetrad systems of Type~II codes of length $24$ using
the theory of harmonic weight enumerators.  This approach
is inspired by Venkov's approach to the classification of the
root systems of Type~II lattices in $\R^{24}$, and gives a new
instance of the analogy between lattices and codes.
\end{abstract}
\maketitle
\section{Background}\label{intro}

We denote by $\F_2$ the two-element field $\Z/2\Z$.
By a ``code'' we mean a \emph{binary linear code of length $n$},
that is, a linear subspace of $\F_2^{n}$.
For such a code~$C$, and any integer~$w$, we define
$$
C_w := \{c \in C: \wt(c)=w \},
$$
where $\wt(c) := |\{ i: c_i = 1 \}|$ is the \emph{Hamming weight}.
Recall that the \emph{dual code} of $C$, denoted $C^\bot$, is defined by
$$
C^\bot :=\{c'\in \F_2^n: \inprod{c}{c'}=0 \text{ for all } c \in C\},
$$
where $\inprod{\cdot}{\cdot}$ is the usual bilinear pairing
$\inprod{x}{y} = \sum_{i=1}^n x_i y_i$ on $\F_2^n$.
We have $\dim(C) + \dim(C^\bot) = n$.
A code $C$\/ is said to be \emph{self-dual} if $C=C^\bot$.
Such a code must have $\dim(C) = n/2$; in particular $2 \mid n$.
Because $\inprod{c}{c} \equiv \wt(c) \bmod 2$, it follows that a self-dual code
$C$\/ is \emph{even}: $2 \mid \wt(c)$ for every word $c \in C$\/;
equivalently, $C_w = \emptyset$ unless $2\mid w$.
A code $C$\/ is said to be \emph{doubly even} if $4 \mid \wt(c)$
for all $c \in C$\/; equivalently, if $C_w=\emptyset$ unless $4\mid w$.

A self-dual code is said to be of \emph{Type II}\/ if it is doubly even,
and of \emph{Type I}\/ otherwise.  Type~II codes are especially rare.
A Type~II code must have length $n=8n'$ for some integer~$n'$,
and for small values of $n'$ the Type~II codes
have been entirely classified.
Indeed, Pless \cite{Pless:GF2} classified all the self-dual codes
of length~$n$, both of Type~I and of Type~II, for $n\leq 20$;
Pless and Sloane \cite{PlessSloane1975}
extended this classification to lengths $n=22$ and~$24$,
citing unpublished work of Conway for the Type~II case;
and Conway and Pless \cite{ConwayPless1980,ConwayPless1992}
classified the Type~II codes of length $32$.
In particular, there is a unique Type~II code of length~$8$
(the $[8,4,4]$ extended Hamming code\footnote{
  Recall that an ``$[n,k,d]$ code'' is a code of length~$n$,
  dimension~$k$, and minimal (nonzero) distance~$d$.
  }), two Type~II codes of length~$16$, and nine of length~$24$.

For sufficiently large $n$, a complete classification of self-dual codes
is likely out of reach.  Indeed, Rains and Sloane
\cite{RainsSloane1998} remarked that ``length $32$ is probably a good
place to stop [seeking such classification results],'' using the
mass formula of MacWilliams, Sloane, and Thompson \cite{MacWilliamsEtAl:Good}
to compute that there are at least $17493$ Type~II codes of length $40$.
King~\cite{King01} further showed that at least $12579$ of these
have $C_4 = \emptyset$, suggesting that it may even be unreasonable
to ask for a classification of \emph{extremal}\footnote{
  Mallows and Sloane~\cite{MallowsSloane} showed that a Type~II code~$C$\/
  of length~$n$ must contain nonzero words of weight at most
  $4\lfloor n/24 \rfloor + 4$ (see also~\cite[p.~194]{SPLAG}).
  If $C_w = \emptyset$ for all positive $w < 4\lfloor n/24 \rfloor + 4$,
  then $C$\/ is said to be ``extremal'': it has the largest minimal
  distance among all Type~II codes of its length.
  }
Type~II codes of length~$40$.

Binary codes~$C$\/ are related with certain \emph{lattices}
$L_C$ in $\R^n$, via the following construction originally due to
Leech and Sloane~\cite{LeechSloane}:
\begin{constra}[{\cite[pp.~182--183]{SPLAG}}]For a code $C\subset \F_{2}^n$,
the lattice $L_C\subset \R^n$ consists of all $x\in \R^n$
such that $2^{1/2} x \in \Z^n$ and $(2^{1/2} x) \bmod 2 \in C$.
\end{constra}
The lattice $L_{C^\bot}$ associated to $C^\bot$
is the dual of $L_C$; that is, it consists of all $x' \in \R^n$
such that $\latinprod{x}{x'} \in \Z$\/ for all $x \in L_C$.\footnote{Here 
$\latinprod{\cdot}{\cdot}$ denotes the standard inner product on $\R^n$.}
In particular, $L_C$ is self-dual if and only if $C$\/ is.
A lattice $L$ for which $\latinprod{x}{x} \in 2\Z$ for all $x \in L$
is said to be \emph{even}; a self-dual lattice is said to be
of \emph{Type~II}\/ if it is even and of \emph{Type~I}\/ if not.
Thus a Construction~A lattice $L_C$ is even if and only if
$C$\/ is doubly even, and if $C$\/ is a Type~I (resp.\ Type~II) code
then $L_C$ is a Type~I (resp.\ Type~II) lattice.
As with codes in $\F_2^n$, Type~II lattices in~$\R^n$ exist
if and only if $8 \mid n$ (see~\cite[p.~53 (Cor.~2)]{Serre:course} and \cite[p.~109]{Serre:course}).  For $n=8$ and $n=16$, Witt~\cite{Witt:8-16} proved
 that the only Type~II lattices are those
of the form $L_C$ for one of the Type~II codes of the same length (see also~\cite[p.~48]{SPLAG}).

Niemeier~\cite{Niemeier:24} was the first to classify the Type~II
lattices of rank~$24$.  There are $24$, including the $9$ lattices
$L_C$ where $C$\/ is one of the Type~II codes of length~$24$.
His technique was later greatly simplified by
Venkov~\cite{Venkov:24} (also in~\cite[Ch.~18]{SPLAG}), who used
weighted theta functions to constrain the possible
root systems of Type~II lattices.  Specifically, Venkov
showed that any rank-$24$ Type~II lattice has one of exactly
twenty-four root systems; the work of Niemeier~\cite{Niemeier:24}
furthermore shows that each of these root systems corresponds
to exactly one rank-$24$ Type~II lattice.\footnote{See 
Ebeling~\cite{Ebeling:lattices} for further discussion.}

Koch~\cite{Koch87} developed a theory of tetrad systems for
codes analogous to the theory of root systems for lattices.
He then obtained a condition on the tetrad systems
of Type~II codes of length $24$ through an appeal to
Venkov's results~\cite{Venkov:24}.  In particular, Koch~\cite{Koch87}
showed that any Type~II code of length $24$ has one of nine tetrad systems;
Conway's classification~\cite{PlessSloane1975} of such codes
furthermore implies that each of these nine arises
for a unique code.

In this paper, we give a new, purely coding-theoretic proof of the
Koch condition~\cite{Koch87} on tetrad systems of Type~II codes of length $24$.
Our method uses Bachoc's theory~\cite{Bachoc:binary,Bachoc:non-binary}
of harmonic weight enumerators,
a coding-theoretic analogue of weighted theta functions, which had
not been developed at the time of Koch's work.
This approach gives a new instance of the analogy between lattices and codes:
our method is directly analogous to that of Venkov~\cite{Venkov:24} for
the classification of the root systems of Type~II lattices in
$\R^{24}$.

The remainder of this paper is organized as follows.  Sections
\ref{tetradintro} and \ref{harmintro} respectively introduce relevant
results from the theories of tetrad systems and harmonic weight enumerators.
Section~\ref{mainthm} states and proves the Koch condition.

\subsection{Tetrad systems}\label{tetradintro}
For a doubly even code $C\subset \F_2^n$,
the set $C_4$ is called the \emph{tetrad system} of~$C$.
In analogy with the theory of root systems for lattices,
the code $\tau(C)$ generated by $C_4$ is called the
\emph{tetrad subcode} of $C$, and if $\tau(C)=C$ then $C$ is called a
\textit{tetrad code}. The irreducible tetrad codes are exactly
\begin{itemize}
  \item the codes $d_{2k}$ ($k \geq 2$), consisting of
   all words $c \in \F_2^{2k}$ of doubly even weight such that
   $c_{2j-1}=c_{2j}$ for each $j=1,2,\ldots,k$\/;
  \item the $[7,3,4]$ dual Hamming code, called $e_7$ in this context; and
  \item the $[8,4,4]$ extended Hamming code, here called $e_8$
\end{itemize}
(see~\cite{Koch87}).  We use the names $d_{2k}$, $e_7$, $e_8$
because the Construction~A lattices $L_{d_{2k}}$, $L_{e_7}$, and $L_{e_8}$
are isomorphic with the root lattices $D_{2k}$, $E_7$, and $E_8$ respectively.

Analogous to the Coxeter number of an irreducible root system,
we define the \emph{tetrad number} $\eta(C)$ of an irreducible
tetrad code~$C$\/ of length~$m$ to be $|C_4|/m$.
A quick computation shows that each of the $m$ coordinates of~$C$\/
takes the value~$1$ on exactly $4\eta(C)$ words in $C_4$, and that
$\eta(d_{2k}) = (k-1)/4$ for each~$k$,
while $\eta(e_7) = 1$ and $\eta(e_8) = 7/4$.

\subsection{Harmonic weight enumerators}\label{harmintro}

Delsarte~\cite{Delsarte:Hahn} introduced the theory of
discrete harmonic polynomials.
For any code $C \subseteq \F_2^n$
and any discrete harmonic polynomial $P: \F_2^n \rightarrow \C$,
the \textit{harmonic weight enumerator} $W_{C,P}(x,y)$ is defined by
\begin{equation}
W_{C,P}(x,y) = \sum_{c\in C} P(c) x^{n-\wt(c)} y^{\wt(c)}
  = \sum_{w=0}^n \left( \sum_{c\in C_w} P(c) \right) x^{n-w} y^{w}.
  \label{harmwtenumgen}
  \end{equation}
These generalized weight enumerators are analogous to the
weighted theta functions of lattice theory; they encode
the distributions of codewords in $C_w$ on the Hamming sphere of radius~$w$.

Bachoc~\cite{Bachoc:binary,Bachoc:non-binary} showed that harmonic weight enumerators
satisfy the following identity generalizing the MacWilliams identity
for Hamming weight enumerators:
$$
W_{C,P}(x,y) =
  \left(-xy\right)^{\deg P}
  \cdot
  \frac{2^{n/2}}{|C^\bot|}
  \cdot
  W_{C^\bot,P}\left(\frac{x+y}{\sqrt 2},\frac{x-y}{\sqrt 2}\right).
$$
Furthermore, by Bachoc~\cite[Cor.~2.1]{Bachoc:binary}, if $C$ is Type~II
and if $P$ is homogeneous of degree~$1$, then $W_{C,P}(x,y)$ is in a 
space of covariant homogeneous polynomials, called
$\mathcal{I}_{\mathcal{G}_1,\chi_1}$ in \hbox{\cite[Lem.~3.1]{Bachoc:binary}}. 
By the same lemma \cite[Lem.~3.1]{Bachoc:binary},
$\mathcal{I}_{\mathcal{G}_1,\chi_1}$ 
contains no nonzero polynomials of degree less than~$30$.

\section{The Koch criterion}\label{mainthm}

\begin{theorem}\label{tetradsystems}
If $C$ is a Type~II code of length $24$, then $C$ has
one of the following nine tetrad systems:
\begin{equation*}\emptyset,\quad
6d_4,\quad4d_6,\quad3d_8,\quad2d_{12},\quad
d_{24},\quad2e_7+d_{10},\quad 3e_8,\quad e_8+d_{16}.\end{equation*}
\end{theorem}
This result follows from the classification~\cite{PlessSloane1975}
of Type~II codes of length~$24$, and also appears
in~\cite{Ebeling:lattices, Koch87}.  In the proof of Theorem
\ref{tetradsystems}, we will use the following proposition.

\begin{prop}\label{irredcomps}
Let $C$ be a Type~II code of length $24$.  Then, \begin{itemize}
\item either $C_4=\emptyset$ or for each $i \in \{1,2,\ldots,24\}$
there exists $c\in C_4$ such that $c_i = 1$.
\item each irreducible component of $\tau(C)$
has tetrad number equal to $|C_4|/24$.
\end{itemize}
\end{prop}

\begin{proof}
For $i=1,2,\ldots,24$ let $P_i$ be the discrete harmonic polynomial
of degree~$1$ given by $P_i(c) = 24 c_i-\wt(c)$, where ``$24 c_i$''
is the real number $0$ or~$24$ according as the $\F_2$ element $c_i$
is $0$ or~$1$.  As we saw at the end of Section~\ref{harmintro}, the 
harmonic weight enumerator $W_{C,P_i}(x,y)$ must vanish for each $i$.
Extracting the $x^{24-w} y^w$ coefficient from~\eqref{harmwtenumgen},
we deduce that
\begin{equation}
\label{sum=0w}
\sum_{c\in C_w}(24c_i-w)=0
\end{equation}
for each $i$ and $w$.
Taking $w = 4$ and reorganizing \eqref{sum=0w} gives
\begin{equation}\label{tetradcond}
\left| \{c \in C_4 : c_i = 1 \} \right| \, = \, |C_4| / 6.
\end{equation}
Thus either $C_4$ is empty or each $i$ is contained in the support
of some $c \in C_4$, and in the latter case
each irreducible component of $\tau(C)$
has tetrad number $\frac14 |C_4|/6 = |C_4| / 24$.
\end{proof}

For any irreducible tetrad code $C$,
the set $C_4$ is a ``$1$-design''\kern-.1ex.\footnote{
  See~\cite{CameronvanLint} or~\cite[p.~88]{SPLAG}
  for an explanation of this terminology.}
Our Proposition~\ref{irredcomps} shows that $C_4$ is also a $1$-design
whenever $C$ is a Type~II code of length $24$.
Theorem~\ref{tetradsystems} now follows:

\begin{proof}[Proof of Theorem~\ref{tetradsystems}]
For each $\eta\not\in\{1,7/4\}$, there is at most one tetrad system
with tetrad number $\eta$.  For each $\eta \in\{1,7/4\}$, there are
two tetrad systems with tetrad number $\eta$, namely $d_{10}$ and $e_7$
for  $\eta = 1$, and $d_{16}$ and $e_8$  for  $\eta = 7/4$.

By Lemma~\ref{irredcomps}, we see that if $C_4$ is nonempty then
either it consists of $m$ tetrad systems of type $d_{2k}$ for some $m$
and $k$ such that $m\cdot 2k = 24$, or it has one of the following forms:
\begin{itemize}
\item $\delta_{10} d_{10} + \varepsilon_7 e_7$,
 with $\varepsilon_7 > 0$ and $10 \delta_{10} + 7 \varepsilon_7 = 24$, or
\item $\delta_{16} d_{16} + \varepsilon_8 e_8$,
 with $\varepsilon_8 > 0$ and  $16 \delta_{16} + 8 \varepsilon_8 = 24$.
\end{itemize}
The list in Theorem~\ref{tetradsystems} then follows almost immediately.
\end{proof}

As we mentioned in Section~\ref{intro}, our approach is directly analogous
to that of Venkov~\cite{Venkov:24}.  We may now explain this analogy explicitly:
in our proof of Theorem~\ref{tetradsystems}, we use the combinatorial
$1$-design property of the tetrad system of $C$\/ in the same way that
Venkov~\cite{Venkov:24} uses the spherical $2$-design property of
the root system of a rank-$24$ Type~II lattice.

\section*{Acknowledgements}
The first author was partly supported by National Science Foundation
grant DMS-0501029.  The second author was partly supported by a
Harvard Mathematics Department Highbridge Fellowship and a Harvard
College Program for Research in Science and Engineering Fellowship.
The authors would like to thank David Hansen for his helpful
comments on an earlier draft of the manuscript.

\bibliographystyle{siam}
\bibliography{references}

\begin{thebibliography}{10}

\bibitem{Bachoc:binary}
{\sc C.~Bachoc}, {\em On harmonic weight enumerators of binary codes}, Designs,
  Codes and Cryptography, 18 (1999), pp.~11--28.

\bibitem{Bachoc:non-binary}
\leavevmode\vrule height 2pt depth -1.6pt width 23pt, {\em Harmonic weight
  enumerators of non-binary codes and macwilliams identities}, in AMS DIMACS
  Series in discrete mathematics and theoretical computer science, A.~Barg and
  S.~Litsyn, eds., vol.~56, American Mathematical Society, 2001, pp.~1--24.

\bibitem{CameronvanLint}
{\sc P.~J. Cameron and J.~H. van Lint}, {\em Designs, Graphs, Codes and their
  Links}, vol.~22 of London Mathematical Society Student Texts, Cambridge
  University Press, 1991.

\bibitem{ConwayPless1980}
{\sc J.~H. Conway and V.~Pless}, {\em On the enumeration of self-dual codes},
  Journal of Combinatorial Theory, Series A, 28 (1980), pp.~26--53.

\bibitem{ConwayPless1992}
\leavevmode\vrule height 2pt depth -1.6pt width 23pt, {\em The binary self-dual
  codes of length up to $32$: a revised enumeration}, Journal of Combinatorial
  Theory, Series A, 60 (1992), pp.~183--195.

\bibitem{SPLAG}
{\sc J.~H. Conway and N.~J.~A. Sloane}, {\em Sphere Packing, Lattices and
  Groups}, Springer-Verlag, 3rd~ed., 1999.

\bibitem{Delsarte:Hahn}
{\sc P.~Delsarte}, {\em {H}ahn polynomials, discrete harmonics, and
  $t$-designs}, SIAM Journal on Applied Mathematics, 34 (1978), pp.~157--166.

\bibitem{Ebeling:lattices}
{\sc W.~Ebeling}, {\em Lattices and Codes: A Course Partially Based on Lectures
  by F. Hirzebruch}, Vieweg, 2nd~ed., 2002.

\bibitem{King01}
{\sc O.~D. King}, {\em The mass of extremal doubly-even self-dual codes of
  length $40$}, IEEE Transactions on Information Theory, 47 (2001),
  pp.~2558--2560.

\bibitem{Koch87}
{\sc H.~Koch}, {\em Unimodular lattices and self-dual codes}, in Proceedings of
  the International Congress of Mathematicians (Berkeley, Calif., 1986),
  American Mathematical Society, 1987, pp.~457--465.

\bibitem{LeechSloane}
{\sc J.~Leech and N.~J.~A. Sloane}, {\em Sphere packing and error-correcting
  codes}, Canadian Journal of Mathematics, 23 (1971), pp.~718--745.

\bibitem{MacWilliamsEtAl:Good}
{\sc F.~J. MacWilliams, N.~J.~A. Sloane, and J.~G. Thompson}, {\em Good
  self-dual codes exist}, Discrete Mathematics, 3 (1972), pp.~153--162.

\bibitem{MallowsSloane}
{\sc C.~L. Mallows and N.~J.~A. Sloane}, {\em An upper bound for self-dual
  codes}, Information and Control, 22 (1973), pp.~188--–200.

\bibitem{Niemeier:24}
{\sc H.-V. Niemeier}, {\em Definite quadratische {F}ormen der {D}imension $24$
  und {D}iskriminante $1$}, Journal of Number Theory, 5 (1973), pp.~142--178.

\bibitem{Pless:GF2}
{\sc V.~Pless}, {\em A classification of self-orthogonal codes over {$GF(2)$}},
  Discrete Mathematics, 3 (1972), pp.~209--246.

\bibitem{PlessSloane1975}
{\sc V.~Pless and N.~J.~A. Sloane}, {\em On the classification and enumeration
  of self-dual codes}, Journal of Combinatorial Theory, Series A, 18 (1975),
  pp.~313--335.

\bibitem{RainsSloane1998}
{\sc E.~M. Rains and N.~J.~A. Sloane}, {\em Self-dual codes}, in Handbook of
  Coding Theory, V.~S. Pless, W.~C. Huffman, and R.~A. Brualdi, eds., Elsevier,
  1998.

\bibitem{Serre:course}
{\sc J.-P. Serre}, {\em A Course in Arithmetic}, Springer-Verlag, 1973.

\bibitem{Venkov:24}
{\sc B.~B. Venkov}, {\em On the classification of integral even unimodular
  24-dimensional quadratic forms}, Proceedings of the Steklov Institute of
  Mathematics, 148 (1980), pp.~63--74.

\bibitem{Witt:8-16}
{\sc E.~Witt}, {\em Eine {I}dentit\"at zwischen {M}odulformen zweiten
  {G}rades}, Abhandlungen aus dem Mathematischen Seminar der Universit\"at
  Hamburg, 14 (1941), pp.~323--337.

\end{thebibliography}

\end{document}